\documentclass[12pt]{article}

\usepackage{amssymb, amsmath, amsthm, color}
\usepackage[onehalfspacing]{setspace}
\usepackage{hyperref}

\definecolor{darkblue}{rgb}{0.0, 0.0, 0.6}
\hypersetup{colorlinks, breaklinks, linkcolor=darkblue, urlcolor=darkblue, anchorcolor=darkblue, citecolor=darkblue}

\newtheoremstyle{plain}
  {3mm}                         
  {3mm}                         
  {\slshape}                    
  {}                            
  {\bfseries}                   
  {.}                           
  {.5em}                        
  {}                            

\theoremstyle{plain}
\newtheorem{thm}{Theorem}[section]

\newtheorem{lem}[thm]{Lemma}

\theoremstyle{definition}
\newtheorem{defn}[thm]{Definition}
\newtheorem{rmk}[thm]{Remark}

\newcommand{\bbN}{\mathbb{N}}

\newcommand{\calB}{\mathcal{B}}
\newcommand{\calJ}{\mathcal{J}}
\newcommand{\calK}{\mathcal{K}}
\newcommand{\calL}{\mathcal{L}}
\newcommand{\calT}{\mathcal{T}}
\newcommand{\la}{\langle}
\newcommand{\ra}{\rangle}
\newcommand{\Pf}{\mathcal{P}_{\!f}}
\newcommand{\setfunc}[2]{\hbox{${}^{\hbox{$#1$}}{\hskip -2 pt #2}$}}
\newcommand{\thmref}[1]{\hyperref[thm:#1]{Theorem~\ref*{thm:#1}}}
\newcommand{\lemref}[1]{\hyperref[lem:#1]{Lemma~\ref*{lem:#1}}}
\newcommand{\secref}[1]{\hyperref[sec:#1]{Section~\ref*{sec:#1}}}
\newcommand{\indicate}[1]{\mathbf{1}_{#1}}

\font\bigmath=cmsy10 scaled \magstep 3
\newcommand{\bigtimes}{\hbox{\bigmath \char'2}}

\renewcommand{\emph}[1]{\textsl{#1}}

\title{A dynamical characterization of $C$ sets}
\author{John H.~Johnson \\
\multicolumn{1}{p{\textwidth}}{%
\centering%
\emph{Department of Mathematics, James Madison University, Harrisonburg, VA 22807, USA}\\
\href{mailto:john.j.jr@gmail.com}{\url{john.j.jr@gmail.com}}
}}

\begin{document}
\maketitle

\begin{abstract}
  Furstenberg, using tools from topological dynamics, defined the notion of a central subset of positive integers, and proved a powerful combinatorial theorem about such sets.
  Using the algebraic structure of the Stone-\v{C}ech compactification, this combinatorial theorem has been generalized and extended to the Central Sets Theorem.
  The algebraic techniques also discovered many sets, which are not central, that satisfy the conclusion of the Central Sets Theorem.
  We call such sets \emph{$C$ sets}.
  Since $C$ sets are defined combinatorially, it is natural to ask if this notion admits a dynamical characterization similar to Furstenberg's original definition of a central set?
  In this paper we give a positive answer to this question by proving a dynamical characterization of $C$ sets.
\end{abstract}

\section{Introduction}
\label{sec:intro}
Furstenberg, in his book connecting dynamical systems with combinatorial number theory, defined the concept of a \emph{central subset of positive integers} \cite[Definition 8.3]{furstenberg:1981} and proved several important properties of such sets, all using notions from topological dynamics.
For instance, whenever a central set is finitely partitioned, at least one cell of the partition contains a central set \cite[Theorem 8.8]{furstenberg:1981}.
Many of the remaining important properties of central sets follows from a powerful combinatorial 
theorem \cite[Proposition 8.21]{furstenberg:1981} also due to Furstenberg.

Inspired by the fruitful interaction between Ramsey Theory and ultrafilters on semigroups, Bergelson and Hindman, with the assistance of B. Weiss, later proved an algebraic characterization of central sets in $\bbN$ \cite[Section 6]{bergelson:1990}.
Using this algebraic characterization as a definition enabled them to easily extend the notion of a central set to any semigroup.
The definition in \cite{furstenberg:1981} also extends naturally to arbitrary semigroups, and the algebraic and dynamical characterizations were proved to be equivalent in general by H. Shi and H. Yang in \cite{shiyang:1996}.

This algebraic definition turns out to have several advantages over the original dynamical definition.
For instance, the fact that central sets are `preserved under finite partitions' (this is a concise way of stating \cite[Theorem 8.8]{furstenberg:1981}) easily follows from the algebraic definition.
More importantly, the combinatorial result \cite[Proposition 8.21]{furstenberg:1981}---and stronger combinatorial statements about central sets---follow from a relatively simple recursive construction.

As an example, we state (currently) the strongest combinatorial theorem about central sets commonly used.
We first state this theorem for commutative semigroups.
(In \secref{csets} of this paper, we shall give the simplest statement of the Central Sets Theorem currently known for arbitrary semigroups.
The statement of this `general' version of the Central Sets Theorem is necessarily complicated because of noncommutativity.)

In the statement of this theorem, and in the remainder of this paper, we let $\Pf(X)$ denote the collection of all nonempty finite subsets of a set $X$; let $\setfunc{A}{B}$ denote the collection of all functions with domain $A$ and codomain $B$; and, for typographical convenience, we let $\calT = \setfunc{\bbN}{S}$ for a given set $S$.
(Generally, the set $S$ in question will be clear from context.)

\begin{thm}[Central Sets Theorem]
  \label{thm:ccst}
  Let $(S, +)$ be a commutative semigroup and $A \subseteq S$ central.
  Then there exist functions $\alpha \colon \Pf(\calT) \to S$ and $H \colon \Pf(\calT) \to \Pf(\bbN)$ that satisfy the following two statements:
  \begin{itemize}
  \item[(1)]
    If $F$, $G \in \Pf(\calT)$ and $F \subsetneq G$, then $\max H(F) < \min H(G)$.
  \item[(2)]
    Whenever $m \in \bbN$, $G_1$, $G_2$, \ldots, $G_m$ is a finite sequence in $\Pf(\calT)$ with $G_1 \subsetneq G_2 \subsetneq \cdots \subsetneq G_m$ and for each $i \in \{1, 2, \ldots, m\}$, $f_i \in G_i$, then we have
    \[
      \sum_{i=1}^m \Bigl( \alpha(G_i) + \sum_{t \in H(G_i)} f_i \Bigr) \in A.
    \]
  \end{itemize}
\end{thm}
\begin{proof}
  This was proved by De, Hindman, and Strauss in \cite[Theorem 2.2]{de:2008}.
\end{proof}
\begin{rmk}
  It's an accident of history that \cite[Proposition 8.21]{furstenberg:1981} is also known in the literature as the Central Sets Theorem.
  Depending on how one counts, there are about four different versions of `the' Central Sets Theorem.
  (Happily each newer version implies or easily reduces to the previous version.)
  From this point on in this paper we shall only refer to \thmref{ccst} and \thmref{cst} as the Central Sets Theorem.
\end{rmk}

We shall call sets that satisfy the conclusion of the Central Sets Theorem \emph{$C$ sets}.
Despite the combinatorial power (and its name), the Central Sets Theorem is \emph{not} strong enough to combinatorially characterize central sets.
In short there are $C$ sets which are not central sets.

This somewhat surprising situation was first discovered in the context of an important type of $C$ sets called the \emph{quasi-central sets}.
The quasi-central sets were first defined algebraically \cite[Definition 1.2]{hindman:1996} and given a combinatorially characterization \cite[Theorem 3.7]{hindman:1996} in a paper of Hindman, Maleki, and Strauss.
The fact that quasi-central sets also satisfy the conclusion of the Central Sets Theorem follows from the proof of \cite[Theorem 2.2]{de:2008}.

Since quasi-central sets are defined algebraically, it is natural to wonder if this notion admits a dynamical characterization similar to Furstenberg's original definition of central sets.
In their recent paper, Burns and Hindman prove such a dynamical characterization of quasi-central sets \cite[Theorem 3.4]{burns:2007}.

However, their paper didn't provide a dynamical characterization of $C$ sets.
(The fact that the notions of $C$ sets and quasi-central sets are distinct follows from an example constructed in a recent paper \cite{hindman:2009} of Hindman.)
In this paper we fill this lacuna and prove a dynamical characterization of $C$ sets in \thmref{csets}.
This characterization will be a special case of a more general result in \thmref{main-result} that gives a dynamical characterization of members of idempotent ultrafilters in compact subsemigroups of the Stone-\v{C}ech compactification.

\subsection*{Acknowledgements}
The characterization proved here is a generalization of part of the author's dissertation research conducted under the guidance of Neil Hindman.
I want to thank Dr.~Hindman for his excellent advisement and helpful comments on this paper.

\section{Preliminaries on Compact Subsemigroups}
\label{sec:compact}

In this section we state the basic definitions, conventions, and results we need to prove our dynamical characterization of members of certain idempotent ultrafilters.
None of the results and definitions in this section are due to the author.
We also omit any proofs, but we do give references to where proofs can be found.

We start by giving a brief review of the algebraic structure of the Stone-\v{C}ech compactification of a discrete semigroup.

Given a discrete nonempty space $S$ we take the points of $\beta S$ to be the collection of all ultrafilters on $S$.
We identify points of $S$ with the principal ultrafilters in $\beta S$.
(Thus we pretend that $S \subseteq \beta S$.)
Given $A \subseteq S$, put $\overline{A} = \{ p \in \beta S : A \in p\}$.
Then the collection $\{ \overline{A} : A \subseteq S \}$ is a basis for a compact Hausdorff topology on $\beta S$.
This topology is the Stone-\v{C}ech compactification of the discrete space $S$.
The proofs for all of these assertions can be found in \cite[Sections 3.2 and 3.3]{hindman:1998}.

Given a discrete semigroup $(S, \cdot)$, we can extend the semigroup operation to $\beta S$ \cite[Theorem 4.1]{hindman:1998} such that for $p$, $q \in \beta S$ and $A \subseteq S$, we have $A \in p \cdot q$ if and only if $\{x \in S : x^{-1}A \in q \} \in p$ \cite[Theorem 4.12]{hindman:1998} where $x^{-1}A = \{ y \in S : xy \in A \}$.
With this operation, $(\beta S, \cdot)$ becomes a compact Hausdorff right-topological semigroup.
The word `right-topological' means that for every $q \in \beta S$ the function $\rho_q \colon \beta S \to \beta S$, defined by $\rho_q(p) = p \cdot q$, is continuous.

\begin{defn}
  Let $S$ be a nonempty discrete space and $\calK$ a filter on $S$.
  \begin{itemize}
  \item[(a)]
    $\overline{\calK} = \{p \in \beta S : \calK \subseteq p \}$.

  \item[(b)]
    $\calL(\calK) = \{ A \subseteq S : S \setminus A \not\in \calK \}$.
  \end{itemize}
\end{defn}

As is well known, the function $\calK \mapsto \overline{\calK}$ is a bijection from the collection of all filters on $S$ onto the collection of all compact subspaces of $\beta S$ \cite[Theorem 3.20]{hindman:1998}.
We also have the following important theorem relating the above two concepts.

\begin{thm}
  \label{thm:uf-extend}
  Let $S$ be a nonempty discrete space and $\calK$ a filter on $S$.
  \begin{itemize}
  \item[(a)]
    $\overline{\calK} = \{ p \in \beta S : \mbox{$A \in \calL(\calK)$ for all $A \in p$} \}$.

  \item[(b)]
    Let $\calB \subseteq \calL(\calK)$ be closed under finite intersections.
    Then there exists a $p \in \beta S$ with $\calB \subseteq p \subseteq \calL(\calK)$.
  \end{itemize}
\end{thm}
\begin{proof}
  Both of these assertions follow from \cite[Theorem 3.11]{hindman:1998}.
\end{proof}

If $(S, \cdot)$ is discrete semigroup and $\calK$ a filter on $S$, then there are precise conditions on $\calK$ which guarantee that $\overline{\calK}$ is a compact subsemigroup of $\beta S$ \cite[Theorem 2.6]{davenport:1990}.
Hence $\overline{\calK}$ is a compact Hausdorff right-topological semigroup for a suitable filter $\calK$.

\begin{thm}
  \label{thm:crts}
  Let $T$ be a compact Hausdorff right-topological semigroup.
    \begin{itemize}
  \item[(a)]
    $T$ contains at least one idempotent, that is, there exists $x \in T$ such that $x = x \cdot x$.

  \item[(b)]
    $T$ contains an ideal, called the smallest ideal and denoted as $K(T)$, that is contained in every ideal of $T$.
    Additionally, $K(T)$ also contains at least one idempotent.
  \end{itemize}
\end{thm}
\begin{proof}
  The proofs of statements (a) and (b) are given in \cite[Theorem 2.5]{hindman:1998} and \cite[Theorem 2.8]{hindman:1998}, respectively.
\end{proof}
\begin{rmk}
  It does follow that $c\ell_{T} K(T)$ is a compact subsemigroup of $T$, and hence by (a) this subsemigroup also contains an idempotent.
  While the smallest ideal $K(T)$ itself may not be closed, it is the union of all of the minimal left ideals of $T$, which are closed, so the fact that it contains an idempotent is also immediate.
\end{rmk}

We are now in a position to give the algebraic definitions of a central set and quasi-central set in a semigroup.

\begin{defn}
  Let $(S, \cdot)$ be a semigroup and $A \subseteq S$.
  \begin{itemize}
  \item[(a)]
    We call $A$ a \emph{central} set if and only if there exists an idempotent $p \in K(\beta S)$ such that $A \in p$.

  \item[(b)]
    We call $A$ a \emph{quasi-central} set if and only if there exists an idempotent $p \in c\ell_{\beta S} K(\beta S)$ such that $A \in p$.
  \end{itemize}
\end{defn}

To finish this section, we give the definition of a dynamical system and relate this notion to the algebraic structure of the Stone-\v{C}ech compactification.

\begin{defn}
  A pair $(X, \la T_s \ra_{s \in S})$ is a \emph{dynamical system} if and only if it satisfies the following four conditions:
  \begin{itemize}
  \item[(1)]
    $X$ is a compact Hausdorff space.
  \item[(2)]
    $S$ is a semigroup.
  \item[(3)]
    $T_s \colon X \to X$ is continuous for every $s \in S$.
  \item[(4)]
    For every $s$, $t \in S$ we have $T_{st} = T_s \circ T_t$.
  \end{itemize}
\end{defn}

\begin{thm}
  \label{thm:dyn-sc}
  Let $(X, \la T_s \ra_{s \in S})$ be a dynamical system.
  Then we can extend $(X, \la T_s \ra_{s \in S})$ to a semigroup action on $\beta S$.
  More precisely, for each $p \in \beta S \setminus S$ we can define $T_p \colon X \to X$ such that for every $q$, $r \in \beta S$, $T_{qr} = T_q \circ T_r$.
  Furthermore, given $p \in \beta S$, $x$ and $y$ in $X$, we have $T_p(x) = y$ if and only if for every neighborhood $U$ of $y$, $\{ s \in S : T_s(x) \in U \} \in p$.
\end{thm}
\begin{proof}
  Both of these assertions follow from \cite[Theorems 3.27 and Corollary 4.22]{hindman:1998}.
\end{proof}

Be warned that for $p \in \beta S \setminus S$, $T_p$ is usually not continuous.

\section{Dynamical Characterization of Members of Idempotent Ultrafilters}
\label{sec:main-result}

\begin{defn}
  Let $(X, \la T_s \ra_{s \in S})$ be a dynamical system, $x$ and $y$ points in $X$, and $\calK$ a filter on $S$.
  The pair $(x, y)$ is called \emph{jointly $\calK$-recurrent} if and only if for every neighborhood $U$ of $y$ we have $\{ s \in S : \mbox{$T_s(x) \in U$ and $T_s(y) \in U$} \} \in \calL(\calK)$.
\end{defn}

\begin{lem}
  \label{lem:technical}
  Let $(X, \la T_s \ra_{s \in S})$ be a dynamical system, let $x$ and $y$ be points in $X$, and let $\calK$ be a filter $S$ such that $\overline{\calK}$ is a compact subsemigroup of $\beta S$.
  The following statements are equivalent.
  \begin{itemize}
  \item[(a)]
    The pair $(x, y)$ is jointly $\calK$-recurrent.
    
  \item[(b)]
    There exists $p \in \overline{\calK}$ such that $T_p(x) = y = T_p(y)$.

  \item[(c)]
    There exists an idempotent $p \in \overline{\calK}$ such that $T_p(x) = y = T_p(y)$.
  \end{itemize}
\end{lem}
\begin{proof}
  \mbox{(a) $\implies$ (b).}
  For each neighborhood $U$ of $y$ put $B_U = \{ s \in S  : \mbox{$T_s(x) \in U$ and $T_s(y) \in U$} \}$.
  Observe that since $B_{U \cap V} = B_U \cap B_V$ for $U$ and $V$ neighborhoods of $y$, we have that the collection $\calB = \{ B_U : \mbox{$U$ is a neighborhood of $y$} \}$ is closed under finite intersections.
  Also, by assumption we have that $\calB \subseteq \calL(\calK)$.
  Hence by \thmref{uf-extend} we can pick $p \in \overline{\calK}$ with $\calB \subseteq p$.

  For every neighborhood $U$ of $y$ we have $B_U \subseteq \{ s \in S : T_s(x) \in U \}$ and $B_U \subseteq \{ s \in S : T_s(y) \in U \}$.
  Therefore $\{ s \in S : T_s(x) \in U \} \in p$ and $\{ s \in S : T_s(y) \in U \} \in p$.
  It now follows from \thmref{dyn-sc} that $T_p(x) = y = T_p(y)$.

  \mbox{(b) $\implies$ (c).}
  Put $M = \{ p \in \overline{\calK} : T_p(x) = y = T_p(y) \}$.
  By \thmref{crts} it suffices to show that $M$ is a compact subsemigroup of $\beta S$.

  $M$ is nonempty by assumption.

  To see that $M$ is compact, we simply show that $M$ is closed.
  Let $p \not\in M$, then either $T_p(x) \ne y$ or $T_p(y) \ne y$.
  First assume that $T_p(x) \ne y$.
  By \thmref{dyn-sc} pick $U$ a neighborhood of $y$ such that $\{ s \in S : T_s(x) \in U \} \not\in p$.
  Put $A = \{ s \in S : T_s(x) \in U \}$ and note that $S \setminus A \in p$.
  We have that $(\overline{S \setminus A}) \cap M = \emptyset$, that is, $\overline{S \setminus A}$ is a (basic) neighborhood of $p$ that misses $M$.
  (If $q \in (\overline{S \setminus A}) \cap M$, then it follows that $A \in q$ and $S \setminus A \in q$, a contradiction.)
  The construction of a (basic) neighborhood of $p$ that misses $M$ when $T_p(y) \ne y$ is similar.
  Therefore $M$ is a closed subset of $\beta S$.

  To see that $M$ is a subsemigroup, let $q$, $r \in M$.
  Then by \thmref{dyn-sc} and assumption we have $ T_{qr}(x) = T_q \circ T_r(x) = T_q(y) = y = T_q \circ T_r(y) = T_{qr}(y)$.
\end{proof}

\begin{thm}[Main Result]
  \label{thm:main-result}
  Let $(S, \cdot)$ be a semigroup, let $\calK$ be a filter on $S$ such that $\overline{\calK}$ is a compact subsemigroup of $\beta S$, and let $A \subseteq S$.
  Then $A$ is a member of an idempotent in $\overline{\calK}$ if and only if there 
exists a dynamical system $(X, \la T_s \ra_{s \in S})$ with points $x$ and $y$ in $X$ and there exists a neighborhood $U$ of $y$ such that the pair $(x,y)$ is jointly $\calK$-recurrent and $A = \{ s \in S : T_s(x) \in U \}$.
\end{thm}
\begin{proof}
  ($\Rightarrow$)
  Let $R = S \cup \{e\}$ be the semigroup with an identity $e$ adjoined to $S$.
  (For expository convenience, we add this new identity even if $S$ already contains an identity.)
  Give $\{0, 1\}$ the discrete topology and give $X = \setfunc{R}{\{0,1\}}$ the product topology.
  Then $X$ is a compact Hausdorff space.

  For each $s \in S$, define $T_s \colon X \to X$ by $T_s(f) = f \circ \rho_s$.
  It's a routine exercise, or see \cite[Theorem 19.14]{hindman:1998}, to show that $(X, \la T_s \ra_{s \in S})$ is a dynamical system.

  Now let $x = \indicate{A}$ be the characteristic function of $A$, pick an idempotent $r \in \overline{\calK}$ with $A \in r$, and put $y = T_r(x)$.
  Then we have that $T_r(y) = T_r\bigl( T_r(x) \bigr) = T_{rr}(x) = T_r(x) = y$.
  Therefore by \lemref{technical} we have that the pair $(x, y)$ is jointly $\calK$-recurrent.

  Put $U = \{ w \in X : w(e) = y(e) \}$ and observe that $U$ is a (subbasic) open neighborhood of $y$.
  (The set $U$ is equal to the inverse image of $\{y(e)\}$ under the projection map.)
  To help us show that $U$ is the neighborhood of $y$ we are looking for, we first will show that $y(e) = 1$.
  Since $y = T_r(x)$ we have that $\{ s \in S : T_s(x) \in U \} \in r$ by \thmref{dyn-sc}.
  We can pick $s \in A$ such that $T_s(x) \in U$.
  Then by definition of $U$ we have that $y(e) = T_s(x)(e) = x \bigl( \rho_s(e) \bigr) = x(es) = x(s)$.
  Also by our choice of $s \in A$ we have $x(s) = \indicate{A}(s) = 1$.

  To finish up this direction observe that for $s \in S$ the following logical relation is true:
  \begin{align*}
    s \in A &\iff \indicate{A}(s) = 1,\\
    &\iff x(s) = 1, \\
    &\iff x(es) = 1, \\
    &\iff (x \circ \rho_s)(e) = 1, \\
    &\iff T_s(x)(e) = 1 = y(e), \\
    &\iff T_s(x) \in U.
  \end{align*}
  Hence $A = \{ s \in S : T_s(x) \in U \}$.

  ($\Leftarrow$)
  Pick a dynamical system $(X, \la T_s \ra_{s \in S})$, pick points $x$ and $y$ in $X$, and pick $U$ a neighborhood of $y$ as guaranteed by assumption.
  By \lemref{technical} pick an idempotent $r \in \overline{\calK}$ such that $T_r(x) = y = T_r(y)$.
  Since $U$ is a neighborhood of $y$ and $T_r(x) = y$, we have $A = \{ s \in S : T_s(x) \in U\} \in r$ by \thmref{dyn-sc}.
\end{proof}
\begin{rmk}
  One remarkable thing about \lemref{technical} and \thmref{main-result} is that, while the results are much more general, the proofs are essentially trivial modifications of the proofs of \cite[Lemma 3.3 and Theorem 3.4]{burns:2007}.
\end{rmk}

\section{A Dynamical Characterization of $C$ sets}
\label{sec:csets}

In this section we give an application of our main result in \secref{main-result} to prove a dynamical characterization of $C$ sets.
We start by giving the combinatorial definition of a $C$ set.
As mentioned in \secref{intro} this combinatorial definition is rather complicated; but we shall soon state an algebraic characterization showing that $C$ sets are members of idempotents in a certain compact subsemigroup.

In the following definition, given an indexed family $\la A_i : i \in I \ra$ of sets, we let $\bigtimes_{i \in I} A_i$ represent its cartesian product.
(We reserve the use of the symbol $\prod$ for our semigroup operation.)
Recall from \secref{intro} that given a set $S$ we let $\calT = \setfunc{\bbN}{S}$ and $\Pf(X)$ is the collection of all nonempty finite subsets of $X$.

\begin{defn}
  Let $(S, \cdot)$ be a semigroup.
  \begin{itemize}
  \item[(a)]
    For each positive integer $m$ put $\calJ_m = \{ (t_1, t_2, \ldots, t_m) \in \bbN^m : t_1 < t_2 < \cdots < t_m \}$.
    
  \item[(b)]
    Given $m \in \bbN$, $a \in S^{m+1}$, $t \in \calJ_m$, and $f \in \calT$, put $x(m, a, t, f) = \prod_{i = 1}^m \bigr( a(i)f(t_i) \bigl)a(m+1)$.
    
  \item[(c)]
    We call a subset $A \subseteq S$ a \emph{$C$ set} if and only if there exist functions $m \colon \Pf(\calT) \to \bbN$, $\alpha \in \bigtimes_{F \in \Pf(\calT)} S^{m(F)+1}$, and $\tau \in \bigtimes_{F \in \Pf(\calT)} \calJ_{m(F)}$ such that the following two statements are satisfied:
    \begin{itemize}
    \item[(1)]
      If $F$, $G \in \Pf(\calT)$ and $F \subsetneq G$, then $\tau(F)\bigl( m(F) \bigr) < \tau(G)(1)$.
      
    \item[(2)]
      Whenever $m \in \bbN$, $G_1$, $G_2$, \ldots, $G_m$ is a finite sequence in $\Pf(\calT)$ with $G_1 \subsetneq G_2 \subsetneq \cdots \subsetneq G_m$, and for each $i \in \{1, 2, \ldots, m\}$, $f_i \in G_i$, then we have
      \[
        \prod_{i=1}^m x(m(G_i), \alpha(G_i), \tau(G_i), f_i) \in A.
      \]
    \end{itemize}
  \end{itemize}
\end{defn}
\begin{rmk}
  This definition of a $C$ set is different from the original (and more complicated) definition given in \cite[Definition 3.3(i)]{de:2008}.
  It is a result in the author's dissertation, to be included in a forthcoming paper \cite{john:2011}, that our simpler definition of a $C$ set is equivalent to the original definition.
\end{rmk}

Before giving an algebraic characterization of $C$ sets, we pause to state the Central Sets Theorem.

\begin{thm}[Central Sets Theorem]
  \label{thm:cst}
  Central sets in a semigroup are $C$ sets.
\end{thm}
\begin{proof}
  This is proved in \cite{john:2011} using the our definition of a $C$ set, and is proved in \cite[Corollary 3.10]{de:2008} under the original definition.
\end{proof}

To give the algebraic definition of a $C$ set we shall need the following (curiously named) combinatorial notion closely related to $C$ sets.

\begin{defn}
  Let $(S, \cdot)$ be a semigroup.
  \begin{itemize}
  \item[(a)]
    We call a subset $A \subseteq S$ a \emph{$J$ set} if and only if for every $F \in \Pf(\calT)$, there exist $m \in \bbN$, $a \in S^{m+1}$, and $t \in \calJ_m$ such that for all $f \in F$, $x(m, a, t, f) \in A$.
    
  \item[(b)]
    $J(S) = \{ p \in \beta S : \mbox{$A$ is a $J$ set for every $A \in p$} \}$.
  \end{itemize}
\end{defn}
\begin{rmk}
  I must point out that $J$ sets are \emph{not} named after the author!
  The term $J$ set is derived from the term `$J_Y$ set' introduced as \cite[Definition 2.4(b)]{hindman:1996} in a different and earlier paper.
  This definition of a $J$ set is also different from the original (and more complicated) definition given in \cite[Definition 3.3(e)]{de:2008}.
  The fact that these two definitions are equivalent is proved by the author in his dissertation and in the forthcoming paper \cite{john:2011}.
\end{rmk}

\begin{thm}
  \label{thm:jset-filter}
  Let $(S, \cdot)$ be a semigroup and $\calK = \{ A \subseteq S : \mbox{$S \setminus A$ is not a $J$ set} \}$.
  Then $\calK$ is a filter on $S$ with $J(S) = \overline{\calK}$ and $J(S)$ is a compact subsemigroup of $\beta S$.
\end{thm}
\begin{proof}
  To show that $\calK$ is nonempty, doesn't contain the empty set, and is closed under supersets is a routine exercise.
  The fact that $\calK$ is closed under finite intersections follows from \cite{john:2011} using the new definition or \cite[Theorem 2.14]{hindman:2010} using the old equivalent definition of $J$ sets.

  Observe that, under the assumption that $\calK$ is a filter, $\calL(\calK) = \{ A \subseteq S : \mbox{$A$ is a $J$ set} \}$.
  Hence it follows from \cite{john:2011} (new definition) or \cite[Theorem 3.11]{hindman:1998} (old equivalent definition) that $J(S) = \overline{\calK}$.

  Finally, the fact that $J(S)$ is a subsemigroup follows from \cite[Theorem 3.5]{de:2008}.
\end{proof}

Since $J(S)$ is a compact subsemigroup, in fact an ideal, of $\beta S$, by \thmref{crts} we have that $J(S)$ contains idempotents.
The following theorem connects these idempotent elements with $C$ sets.

\begin{thm}
  \label{thm:cset-id}
  Let $(S, \cdot)$ be a semigroup and $A \subseteq S$.
  Then $A$ is a $C$ set if and only if there exists an idempotent $p \in J(S)$ such that $A \in p$.
\end{thm}
\begin{proof}
  This is proved in \cite{john:2011} for the new definition or \cite[Theorem 3.8]{de:2008} for the old definition.
\end{proof}

Using these facts and the general results in \secref{main-result} we end with the following dynamical characterization of $C$ sets.

\begin{thm}
  \label{thm:csets}
  Let $(S, \cdot)$ be a semigroup and $A \subseteq S$.
  Then $A$ is a $C$ set if and only if there exists a dynamical system $(X, \la T_s \ra_{s \in S})$ with points $x$ and $y$ in $X$ and there exists a neighborhood $U$ of $y$ such that $\{ s \in S : \mbox{$T_s(x) \in U$ and $T_s(y) \in U$} \}$ is a $J$ set and $A = \{ s \in S : T_s(x) \in U \}$.
\end{thm}
\begin{proof}
  Let $\calK = \{ B \subseteq S : \mbox{$S \setminus B$ is not a $J$ set} \}$ and note by \thmref{jset-filter} that $\overline{\calK} = J(S)$ and $\calL(\calK) = \{ A \subseteq S : \mbox{$A$ is a $J$ set} \}$.
  Since \thmref{cset-id} characterizes $C$ sets in terms of idempotents in $\overline{\calK}$, we can apply \thmref{main-result} to prove our statement.
\end{proof}

\bibliographystyle{amsplain}
\bibliography{references}  
\end{document}